\documentclass{article}

\usepackage{amsmath,amsthm,amssymb,amsfonts}
\usepackage{amscd}
\usepackage{verbatim}

\usepackage{hyperref}

\usepackage[final]{pdfpages}

\usepackage[colorinlistoftodos]{todonotes}

\usepackage{tikz}
\usepackage{tikz-cd}

\theoremstyle{plain}
\newtheorem{thm}{Theorem}
\newtheorem{lem}[thm]{Lemma}
\newtheorem{prop}[thm]{Proposition}
\newtheorem{cor}[thm]{Corollary}

\theoremstyle{definition}

\newtheorem{exl}[thm]{Example}
\newtheorem{question}[thm]{Question}

\numberwithin{thm}{section}

\DeclareMathOperator{\Fix}{Fix}
\DeclareMathOperator{\Coin}{Coin}
\DeclareMathOperator{\MF}{MF}

\def\Z{{\mathbb Z}}
\def\R{{\mathbb R}}

\def\vv{\mathbf v}
\def\ve{\mathbf e}
\def\vx{\mathbf x}
\def\vy{\mathbf y}
\def\vc{\mathbf c}

\begin{document}
\title{Partitions of $n$-valued maps}
\author{P. Christopher Staecker}

\maketitle

\begin{abstract}
An $n$-valued map is a set-valued continuous function $f$ such that $f(x)$ has cardinality $n$ for every $x$. Some $n$-valued maps will ``split'' into a union of $n$ single-valued maps. Characterizations of splittings has been a major theme in the topological theory of $n$-valued maps. 

In this paper we consider the more general notion of ``partitions'' of an $n$-valued map, in which a given map is decomposed into a union of other maps which may not be single-valued. We generalize several splitting characterizations which will describe partitions in terms of mixed configuration spaces and mixed braid groups, and connected components of the graph of $f$.
We demonstrate the ideas with some examples on tori. 

We also discuss the fixed point theory of $n$-valued maps and their partitions, and make some connections to the theory of finite-valued maps due to Crabb.
\end{abstract}

\section{Introduction}
Given sets $X$ and $Y$ and a positive integer $n$, an $n$-valued function from $X$ to $Y$ is a set-valued function $f$ on $X$ such that $f(x)\subseteq Y$ has cardinality exactly $n$ for every $x\in X$. Equivalently, an $n$-valued function on $X$ is a single-valued function $f:X \to D_n(Y)$, where $D_n(Y)$ is the \emph{unordered configuration space of $n$ points in $Y$}, defined as:
\[ D_n(Y) = \{ \{y_1,\dots,y_n\} \mid y_i \in Y, y_i \neq y_j \text{ for } i\neq j \}. \]

When $Y$ is a topological space, we give $D_n(Y)$ a topology as follows: begin with the product topology on $Y^n$, then consider the subspace $F_n(Y)$ of tuples $(y_1,\dots,y_n)\in Y^n$ with $y_i \neq y_j$ for $i\neq j$. This is the \emph{ordered configuration space}. Then $D_n(Y)$ is the quotient of $F_n(Y)$ up to ordering, and so its topology is given by the quotient topology. When $f:X\to D_n(Y)$ is continuous, we call it an \emph{$n$-valued map from $X$ to $Y$}. Continuity of $n$-valued maps can also be defined in terms of lower- and upper-semicontinuity. These approaches are equivalent: see \cite{bg18}.

An $n$-valued map $f:X\to D_n(Y)$ is \emph{split} if there are $n$ single valued continuous functions $f_1,\dots,f_n:X\to Y$ with $f(x) = \{ f_1(x),\dots,f_n(x) \}$ for every $x\in X$. In this case we write $f= \{f_1,\dots, f_n\}$, and we say that $\{f_1,\dots, f_n\}$ is a \emph{splitting} of $f$. Not all $n$-valued maps are split. 

In this paper we wish to generalize the notion of splitting to the case where the various $f_i$ are not necessarily single-valued. If $k\le n$ and $g:X\to D_k(X)$ is a $k$-valued map with $g(x) \subseteq f(x)$ for every $x$, then we say $g$ is a \emph{submap} of $f$. When $k<n$ we say $g$ is a \emph{proper submap}. If $f$ has no proper submap, then we say $f$ is \emph{irreducible}.

Let $f$ be an $n$-valued map and let $k_1,\dots,k_l$ be a partition of $n$, that is, positive integers which sum to $n$, and assume that there is a set of $l$ maps $f_1,\dots,f_l$, where $f_i$ is a $k_i$-valued map and $f(x) = f_1(x) \cup \dots \cup f_l(x)$ for each $x$. Then we say $\{f_1,\dots,f_l\}$ is a \emph{partition} of $f$, and we write $f=\{f_1,\dots,f_l\}$. If each $k_i = 1$, then this is a splitting. When each map $f_i$ is irreducible, we say $\{f_1,\dots,f_l\}$ is a \emph{partition of $f$ into irreducibles}. 

Consider a simple example:
\begin{exl} Let $f:S^1 \to D_n(S^1)$ be the 4-valued circle map given by 
\[ f(x) = \left\{\frac x 2, \frac x2+ \frac14, \frac x2+\frac12, \frac x2+\frac34\right\}. \] 
where $S^1$ is parameterized as real numbers read modulo 1. In the terminology of \cite{brow06}, $f$ is the linear 4-valued circle map of degree 2. The \emph{graph} of $f$, defined by $\{(x,y) \mid y \in f(x)\}$, looks like:
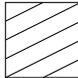
\[ 
\begin{tikzpicture}
\draw (0,0) rectangle (1,1);
\foreach \x in {0,1/4,1/2} {
	\draw (0,\x) -- (1,\x+1/2);
	}
\draw (0,3/4) -- (1/2,1);
\draw (1/2,0) -- (1,1/4);
\end{tikzpicture}
 \]
In this example, $f$ partitions into two 2-valued maps: 
\[ f_1(x) = \left\{\frac x2, \frac x2 + \frac12\right\},\quad f_2(x) = \left\{\frac x2+\frac14, \frac x2+\frac34\right\}, \] 
and this is a partition into irreducibles.
\end{exl}

We will discuss partitions of $n$-valued maps from several viewpoints: In Section \ref{graphsection} we recall existing work which relates partitions to the connected components of the graph of $f$. In Section \ref{configsection} we relate partitions to liftings of the map to mixed configuration spaces. 
In Section \ref{reidsection} we show how partitions relate to the theory of lifting classes in \cite{bdds20}, and in Section \ref{torisection} we show some results concerning partitions of linear $n$-valued maps on tori.   In Section \ref{jiangsection} we extend some work from \cite{bdds20} concerning the Jiang property for $n$-valued maps. We conclude in Section \ref{crabbsection} by pointing out some connections to the work of Crabb on the fixed point theory of finite-valued maps.

\section{Partitions and the graph of $f$}\label{graphsection}
In this paper all spaces will be assumed to be finite polyhedra.
For any function $f:X \to D_n(Y)$, let $\Gamma_f \subset X\times Y$ be the \emph{graph of $f$}, defined by:
\[ \Gamma_f = \{ (x,y) \in X\times X \mid y \in f(x) \}. \]

Robert F. Brown, in \cite{brow08}, studied partitions of $n$-valued maps, which he calls ``$w$-splittings''. Brown's Proposition 2.1 states that when $\Gamma_f$ has $k$ path components, then $f$ has a partition into $k$ maps. In fact Brown's proof implies the following more specific statement:
\begin{thm}[\cite{brow08}, Proposition 2.1]\label{graphpartition}
For any map $f:X\to D_n(Y)$, there is an irreducible partition $f = \{f_1,\dots,f_k\}$ if and only if $\Gamma_f$ has $k$ path components. In this case, the components of $\Gamma_f$ are exactly the sets $\Gamma_{f_k}$. 
\end{thm}

It is well known that splittings, when they exist, are unique. 
Since $\Gamma_f$ has a unique decomposition into its path components, we obtain the following generalization:
\begin{cor}
Any map $f:X\to D_n(Y)$, has a unique irreducible partition as $f=\{f_1,\dots,f_k\}$, where perhaps $k=1$ and the partition is unique up to the ordering of the $f_i$.
\end{cor}

\section{Partitions and mixed configuration spaces}\label{configsection}

A basic criterion for the existence of a splitting involves the covering of the ordered configuration space over the unordered configuration space. As discussed in \cite[Section 2.1]{gg18}, the map $r:F_n(Y)\to D_n(Y)$ which forgets the ordering of an $n$-element configuration is a covering map, and a map $f:X\to D_n(Y)$ splits if and only if it has a lifting $\hat f:X\to F_n(Y)$. In this section we generalize this result for partitions.

As in \cite{gg04}, p.131, if $(d_1,\dots,d_k)$ is a partition of $n$ (that is, the $d_i$ are positive integers summing to $n$), let $D_{d_1,\dots,d_k}(Y)$ be the quotient of $F_n(Y)$ by $\Sigma_{d_1}\times \dots \times \Sigma_{d_k}$, where $\Sigma_d$ is the symmetric group on $d$ elements. This is the ``mixed configuration space'' of $n$ points, in which certain subsets of the $n$ points are ordered. In particular when $k=2$ and $n = m + l$ we have 
\[ D_{m,l} (Y) = \{ (\{y_1,\dots, y_m\}, \{y_{m+1},\dots,y_n\}) \mid y_i \neq y_j \text{ for } i\neq j \} \]

\begin{prop}
For $m+l=n$, there are coverings $s_{m,l}:F_n(Y) \to D_{m,l}(Y)$ and $t_{m,l}: D_{m,l}(Y)\to D_n(Y)$ with $t_{m,l} \circ s_{m,l} = r$.
\end{prop}
\begin{proof}
The required coverings are simply:
\begin{gather*}
s_{m,l}(y_1,\dots,y_n) = (\{y_1,\dots,y_m\},\{y_{m+1},\dots,y_n\}) \\
t_{m,l} (\{y_1,\dots,y_m\},\{y_{m+1},\dots,y_n\}) = \{y_1,\dots,y_n\}. \qedhere
\end{gather*}
\end{proof}

The following theorem generalizes the splitting criterion discussed above.
\begin{thm}\label{partlift}
Let $f:X\to D_n(Y)$ be an $n$-valued map. Then $f$ has a proper partition if and only if there are positive integers $m$ and $l$ with $m+l=n$ and $f$ lifts to a map $\hat f:X \to D_{m,l}(Y)$.
\end{thm}
\begin{proof}
First assume that $f$ has a proper partition, say $f = g \cup h$ where $g:X\to D_m(Y)$ and $h:X \to D_l(Y)$. Then let $\hat f$ be defined by $\hat f(x) = (g(x),h(x))$. Since $f = g\cup h$, the sets $g(x)$ and $h(x)$ are disjoint for all $x$, and thus $\hat f$ is the desired lift $\hat f:X \to D_{m,l}(Y)$. 

For the converse, assume that $f$ has a lift $\hat f:X\to D_{m,l}(Y)$. This $\hat f$ will be a pair of functions of the form $f(x) = (g(x),h(x))$ where $g:X\to D_m(Y)$ and $h:X \to D_l(Y)$, and this $g$ and $h$ provide the desired partition.
\end{proof}

The fundamental group $\pi_1(D_n(X))$ is the $n$-strand braid group $B_n(X)$. 
When $(d_1,\dots,d_k)$ is a partition of $n$, let $B_{d_1,\dots,d_k}(Y) = \pi_1(D_{d_1,\dots,d_k}(Y))$. This is the \emph{mixed braid group} on $Y$ first defined in \cite{manf97}, the subgroup of $B_n(Y)$ in which the strands starting in the first $d_1$ positions must end in the first $d_1$ positions, the strands starting in the next $d_2$ positions must end in the next $d_2$ positions, etc. When all $d_i=1$, we have $B_{1,\dots,1}(Y) = P_n(Y) = \pi_1(F_n(Y))$, the pure braid group on $Y$. When $k=1$ and thus $d_1=n$, we obtain simply $B_n(Y)$, the full braid group on $Y$.

Our next theorem uses the ``lifting criterion'' from covering space theory: Let $p:\bar Y \to Y$ be a covering space and $f:X \to Y$ be a map with $X$ path connected and locally path connected. Then a lift $\bar f:X\to \bar Y$ exists if and only if $f_\#(\pi_1(X)) \subseteq p_\#(\pi_1(\bar Y))$. See \cite{hatc02}, Proposition 1.33.

The following is a generalization of the ``Splitting Characterization Theorem'', Theorem 3.1 of \cite{bg18}, which states that $f$ splits if and only if $f_\#(\pi_1(X))$ is a subgroup of $P_n(Y)$, where $\#$ denotes the induced homomorphism on the fundamental group. Recall the covering map $t_{m,l}: D_{m,l}(Y) \to D_n(Y)$ which forgets the ordering in $D_{m,l}(Y)$. Its induced homomorphism in fundamental groups is $t_{m,l\#}:B_{m,l}(Y) \to B_n(Y)$, which is the inclusion.
\begin{thm}
Let $f:X\to D_n(Y)$ be an $n$-valued map. Then $f$ has a proper partition if and only if there are positive integers $m$ and $l$ with $m+l=n$ such that $f_\#(\pi_1(X))$ is a subgroup of $t_{m,l\#}(B_{m,l}(Y))$.
\end{thm}
\begin{proof}
By Theorem \ref{partlift}, $f$ has a proper partition if and only if there are $m,l$ with $m+l=n$ and $f$ lifts to some $\hat f:X \to D_{m,l}(Y)$. By the lifting criterion, this is equivalent to $f_\#(\pi_1(X)) \subseteq t_{m,l\#}(B_{m,l}(Y))$.
\end{proof}

\section{Partitions in fixed point theory}\label{reidsection}
A \emph{fixed point} of such a map is a point $x\in X$ with $x \in f(x)$, and we denote the set of fixed points of $f$ by $\Fix(f)\subset X$. Topological fixed point theory of $n$-valued maps was first studied by Schirmer in \cite{schi84a, schi84b}. 

For the rest of the paper we will focus on selfmaps $f:X\to D_n(X)$. First we briefly review the theory of lifting classes and Reidemeister classes for $n$-valued maps which is developed in \cite{bdds20}. Given the universal covering $p:\tilde X \to X$, the \emph{orbit configuration space} with respect to this cover is:
\[ F_n(\tilde X, \pi) = \{ (\tilde x_1,\dots,\tilde x_n) \in \tilde X^n \mid p(\tilde x_i) \neq p(\tilde x_j) \text{ for $i \neq j$} \}, \]
and the map $p^n:F_n(\tilde X,\pi) \to D_n(X)$ given by applying $p$ to each coordinate is a covering map.

Given a map $f:X\to D_n(X)$ and some choices of basepoints for $\tilde X$ and $F_n(\tilde X,\pi)$, there is a well-defined \emph{basic lifting} $\bar f: \tilde X \to F_n(\tilde X,\pi)$ so that the diagram commutes:
\[ 
\begin{tikzcd}
\tilde X \arrow[r,"\bar f"] \arrow[d,"p"] & F_n(\tilde X,\pi) \arrow[d,"p^n"] \\
X \arrow[r,"f"] & D_n(X)
\end{tikzcd}  
\]
This basic lifting $\bar f$ is a lifting of $f$ with respect to the covering $p^n:F_n(\tilde X,\pi) \to D_n(X)$. The covering group of $F_n(\tilde X,\pi)$ is $\pi_1(X)^n \rtimes \Sigma_n$, where $\Sigma_n$ is the symmetric group on $n$ elements. Given $\alpha \in \pi_1(X)$, we have the basic lift $\bar f$, and viewing $\alpha$ as a covering transformation on $\tilde X$ gives $\bar f \circ \alpha$, which is some other lift of $f$. 

Lemma 2.5 of \cite{bdds20} describes how any $n$-valued map $f$ determines a homomorphism $\psi_f:\pi_1(X) \to \pi_1(X)^n\rtimes \Sigma_n$ by the formula:
\[ \psi_f(\alpha) \circ \bar f = \bar f \circ \alpha. \]

Writing $\psi_f$ in coordinates as $\psi_f(\alpha) = (\phi_1(\alpha),\dots,\phi_n(\alpha);\sigma(\alpha))$, we obtain functions $\phi_i:\pi_1(X) \to \pi_1(X)$ and $\sigma:\pi_1(X)\to \Sigma_n$ such that (writing $\sigma(\alpha) = \sigma_\alpha$):
\[ \bar f_i(\alpha \tilde x) = \phi_i(\alpha)\bar f_{\sigma^{-1}_{\alpha}(i)}(\tilde x). \]
The functions $\phi_i:\pi_1(X) \to \pi_1(X)$ are not necessarily homomorphisms, but $\sigma$ is.

Given some $i,j \in \{1,\dots,n\}$ and $\alpha,\beta\in \pi_1(X)$, the pairs $(\alpha,i)$ and $(\beta,j)$ are \emph{Reidemeister equivalent}, and we write $[(\alpha,i)] = [(\beta,j)]$, when there is some $\gamma\in \pi_1(X)$ with $\sigma_\gamma(j)=i$ and 
\[ \alpha = \gamma \beta \phi_i(\gamma^{-1}). \]
This is an equivalence relation, and the number of these equivalence classes is the \emph{Reidemeister number} $R(f)$. Any set of the form $p(\Fix(\alpha \bar f_i))$ for $\alpha\in \pi_1(X)$ is called a \emph{fixed point class}, and by \cite[Theorem 2.9]{bdds20} two nonempty fixed point classes $p(\Fix(\alpha \bar f_i))$ and $p(\Fix(\beta \bar f_j))$ are equal if and only if $[(\alpha,i)] = [(\beta,j)]$, and are disjoint when $[(\alpha,i)] \neq [(\beta,j)]$.  

For $i,j \in \{1,\dots,n\}$, we will write $i\sim j$ when there is some $\gamma \in \pi_1(X)$ with $\sigma_\gamma(j) = i$.

\begin{prop}
The relation $\sim$ is an equivalence relation. 
\end{prop}
\begin{proof}
The required properties will follow from the fact that $\sigma$ is a homomorphism.

For reflexivity, since $\sigma$ is a homomorphism we have $\sigma_1(i)=i$ and thus $i\sim i$.

For symmetry, let $i\sim j$ with $\sigma_\alpha(j) = i$. Then $\sigma_{\alpha^{-1}}(i) = \sigma_\alpha^{-1}(i) = j$ and so $j \sim i$. 

For transitivity, let $i\sim j$ and $j\sim k$ with $\sigma_{\alpha}(j)=i$ and $\sigma_{\beta}(k)=j$. Then:
\[ \sigma_{\alpha\beta}(k) = \sigma_\alpha(\sigma_\beta(k)) = \sigma_\alpha(j) = i \]
and so $i\sim k$.
\end{proof}

The relation $\sim$ divides the basic lift $\bar f = \{\bar f_1,\dots,\bar f_n\}$ into equivalence classes where $\bar f_i$ and $\bar f_j$ are in the same class when $i \sim j$. These classes are called the \emph{$\sigma$-classes} of the basic lift. 

The next two results relate partitions of $f$ to the structure of the $\sigma$-classes.

\begin{thm}
Let $f:X\to D_n(X)$ be an $n$-valued map, and let $\bar f:\tilde X \to F_n(\tilde X,\pi)$ be the basic lift.
Let $\bar g = \{\bar f_{i_1},\dots,\bar f_{i_k}\} \subset \{\bar f_1,\dots,\bar f_n\} = \bar f$ be a $\sigma$-class. Then there is a $k$-valued submap of $f$ which lifts to $\bar g$. 
\end{thm}
\begin{proof}
We define $g:X\to D_k(X)$ by $g(x)= p^k(\bar g(\tilde x))$, where $\tilde x \in \tilde X$ is some point with $p(\tilde x) = x$. This $g$ is a $k$-valued submap of $f$ which lifts to $\bar g$. We need only show that $g$ is well-defined with respect to the choice of point $\tilde x \in p^{-1}(x)$. That is, we must show that $p^k(\bar g(\tilde x)) = p^k(\bar g(\alpha \tilde x))$ for any $\alpha \in \pi_1(X)$. 

Since $\{\bar f_{i_1},\dots,\bar f_{i_k}\}$ is a $\sigma$-class, we will have $\sigma_\alpha(\{i_1,\dots,i_k\}) \subseteq \{i_1,\dots,i_k\}$. And since $\sigma_\alpha$ is a permutation and these are both sets of $k$ elements, we will have $\sigma_\alpha(\{i_1,\dots,i_k\}) = \{i_1,\dots,i_k\}$. Thus we have
\begin{align*} 
p^k(\bar g(\alpha\tilde x)) &= \{p(\bar f_{i_1}(\alpha\tilde x)), \dots, p(\bar f_{i_k}(\alpha\tilde x))\} \\
&= \{p(\phi_{i_1}(\alpha)\bar f_{\sigma^{-1}_\alpha(i_1)}(\tilde x)), \dots, p(\phi_{i_k}(\alpha)\bar f_{\sigma^{-1}_\alpha(i_k)}(\tilde x)) \} \\
&= \{p(\bar f_{\sigma^{-1}_\alpha(i_1)}(\tilde x)), \dots, p(\bar f_{\sigma^{-1}_\alpha(i_k)}(\tilde x)) \} \\
&= \{p(\bar f_{i_1}(\tilde x)), \dots, p(\bar f_{i_k}(\tilde x)) \} = p^k(\bar g(\tilde x)),
\end{align*}
as desired.
\end{proof}

\begin{thm}\label{submaplift}
Let $g:X\to D_k(X)$ be a submap of $f:X\to D_n(X)$, and let $\bar f = (\bar f_1,\dots,\bar f_n)$ be the basic lifting of $f$. Then $g$ has a lifting of the form $\bar g = (\bar f_{i_1},\dots,\bar f_{i_k})$ for some indices $i_j \in \{1,\dots,n\}$. 

Furthermore, if $g$ is irreducible, then $\{\bar f_{i_1},\dots,\bar f_{i_k}\}$ is a $\sigma$-class of $\bar f$. 
\end{thm}
\begin{proof}
For the first statement, let $x_0\in X$ be the basepoint, and we have $g(x_0) \subset f(x_0)$. Let $\tilde x_0 \in p^{-1}(x_0)$ be the basepoint, and let $i_j$ be chosen so that $p^k(\bar f_{i_1}(\tilde x_0),\dots,\bar f_{i_k}(\tilde x_0)) = g(x_0)$. Then clearly $\bar g = (\bar f_{i_1},\dots,\bar f_{i_k})$ is a lifting of $g$. 

For the second statement, we will prove the contrapositive. Assume that $\{\bar f_{i_1},\dots,\bar f_{i_k}\}$ is not a $\sigma$-class, and we will show that $g$ is reducible. Since $\{\bar f_{i_1},\dots,\bar f_{i_k}\}$ is not a single $\sigma$-class, it has a nontrivial partition into $\sigma$-classes. Without loss of generality assume that $\{\bar f_{i_1},\dots,\bar f_{i_l}\}$ is a $\sigma$-class for some $l<k$. 

Now let $h = p^l\{\bar f_{i_1},\dots,\bar f_{i_l}\}$, and we have $h(x) \subset g(x)$ and $h(x)$ has cardinality $l<k$ for each $x$. Thus $g$ is reducible, and we have proved the contrapositive of the second statement.
\end{proof} 

Combining the two theorems above shows that the partitions of $f$ are exactly determined by the $\sigma$-classes:

\begin{thm}\label{sigmaclasspartition}
There is a bijective correspondence between the set of irreducible submaps of $f$ and the set of $\sigma$-classes of the $\bar f_i$.
\end{thm}


The next theorem shows that fixed point classes of $f$ naturally respect partitions of $f$.

\begin{thm}
Let $f$ be an $n$-valued map with a partition $f =\{g,h\}$. Then every fixed point class of $g$ is disjoint from every fixed point class of $h$.
\end{thm}

\begin{proof}
Since $g$ and $h$ are submaps of $f$, each fixed point class of $g$ has the form $p\Fix(\alpha\bar f_i)$ for some $\alpha\in \pi_1(X)$, and each fixed point class of $h$ has the form $p\Fix(\beta\bar f_j)$ for some $\beta \in \pi_1(X)$. Furthermore, the lifts $\bar f_i$ and $\bar f_j$ belong to different $\sigma$-classes, since $g$ and $h$ do not contain a common irreducible submap. Thus there is no $\gamma$ with $\sigma_\gamma(j)=i$, and thus $[(\alpha,i)] \neq [(\beta,j)]$ and so $p\Fix(\alpha\bar f_i) \cap p\Fix(\beta\bar f_j) = \emptyset$.
\end{proof}

Since the Nielsen number is the number of essential fixed point classes, we immediately obtain:
\begin{cor}\label{Npartition}
If $f$ has a partition $f = \{g,h\}$, then $N(f) = N(g) + N(h)$. 
\end{cor}

\begin{cor}
If $f = \{ f_1,\dots,f_k\}$ is the partition into irreducibles, then 
\begin{equation}\label{Nparteq}
 N(f) = N(f_1) + \dots + N(f_k). 
 \end{equation}
\end{cor}



\section{Circle maps and linear maps on tori}\label{torisection}
The $n$-valued Nielsen theory on tori has been studied in detail for the class of \emph{linear $n$-valued maps}, introduced in \cite{bl10}. This class of maps includes all $n$-valued maps on the circle (up to homotopy). 

Denote the 
universal covering space of the torus $T^q$ by $p^q \colon 
\R^q \to T^q$ where $p(t) = t \mod 1$.  A $q \times q$
integer matrix $A$ induces a map $f_A \colon \R^q/\Z^q = T^q \to T^q$ by
$$
f_A(p^q(\vv)) = p^q(A\vv) = (p(A_1 \cdot \vv), \dots , p(A_q \cdot \vv)),
$$
where $A_j$ is the $j$-th row of $A$. 

Given $\vx = (x_1, \dots , x_q), \vy = (y_1, \dots 
, y_q) \in \mathbb R^q$, we say that $\vx = \vy \mod n$ when $x_j = y_j \mod n$ for each $j$. Let $\vc$ be the vector whose coordinates all equal $1$.
Define $f^{k}_{n, A} 
\colon \mathbb R^q \to \mathbb R^q$ by
$$
f^{k}_{n, A}(t) = \frac 1n(At + k\vc).
$$

\begin{thm} (\cite{bl10}, Theorem 3.1, Theorem 4.2) A $q \times q$
integer matrix $A$ induces an $n$-valued map
$f_{n, A} \colon T^q \to D_n(T^q)$ defined by
$$
f_{n, A}(p^q(t)) = p^q\{ f^{1}_{n, A}(t), \dots , 
f^{n}_{n, A}(t)\}
$$
if and only if $A_i = A_j \mod n$ for all
$i, j \in \{1, \dots , q\}$.

For this map, we have $N(f_{n,A}) = n|\det(I-\frac1n A)|$, where $I$ is the identity matrix.
\end{thm}

In the case $q=1$, our space is $T^1=S^1$, the circle. The theory of $n$-valued maps on the circle was extensively studied by Brown in \cite{brow06}.
Brown showed that any $n$-valued map $f:S^1 \to D_n(S^1)$ is homotopic to some linear map. The $q\times q$-matrix $A$ for this linear map is a single integer, called the \emph{degree} of $f$, and the formula for the Nielsen number is simply $N(f) = |n-d|$. 

Brown also showed that $f_{n,d}:S^1 \to D_n(S^1)$, the linear circle map of degree $d$, is split if and only if $n\mid d$. In that case, $f$ splits into $n$ maps, each homotopic to the map of degree $d/n$. We generalize this result to partitions of linear torus maps as follows:
\begin{thm}\label{linearpartition}
Let $f_{n,A}:T^q \to D_n(T^q)$ be a linear torus map induced by some $q\times q$-matrix $A$, such that, for each row $A_j$, we have $A_j = [l_1,\dots, l_q] \mod n$. Then $f_{n,A}$ has a nontrivial partition if and only if there is some $m\in \Z$ with $m> 1$ which is a common factor of $n$ and every $l_j$. 

In this case, $f_{n,A}$ partitions into a a set of $m$ maps, each homotopic to the linear $\frac nm$-valued map $f_{\frac nm, \frac{1}nA}$. 
\end{thm}
\begin{proof}
We will compute the $\sigma$-orbits of $f_{n,A}$. Let $\{\ve_1,\dots,\ve_q\}$ be the standard basis vectors for $\R^q$. Since $\sigma:\pi_1(T^q) \to \Sigma_n$ is a homomorphism, it suffices to compute $\sigma_{\ve_j}$ for each $j$. Since every row $A_j$ of $A$ has the form $A_j = [l_1,\dots,l_q] \mod n$, every column $A\ve_j$ has every entry equal to $l_j \mod n$, that is, $A\ve_j = l_j\vc \mod n$.

We have:
\begin{align*}
f^k_{n,A}(\vv + \ve_j) &= \frac{1}{n}(A(\vv+\ve_j)+k\vc) = \frac1n(A \vv + A\ve_j + k\vc) \\
&= \frac1n(A\vv + (k+l_j)\vc) \mod 1 = f^{k+l_j}_{n,A}(\vv) \mod 1
\end{align*}
and we see that $\sigma_{\ve_j} (k)= k + l_j$. Thus the $\sigma$-class of $k$ is all of $\{1,\dots,n\}$ if and only if some $l_j$ has no nontrivial common factor with $n$, and we have proved the first statement of the theorem.

For the second statement, assume that $m>1$ is some nontrivial common factor of $n$ and every $l_j$. Then the $\sigma$-orbits in $\{1,\dots,n\}$ are simply the conjugacy classes modulo $n/m$. Thus there are $n/m$ different $\sigma$-orbits, each consisting of $m$ elements. Then $f_{n,A}$ partitions into $m$ maps, each being $n/m$-valued. Given some $k \in \{1,\dots,n\}$, divide $k$ by $m$ to obtain a quotient and remainder $k = sm + r$ with $s\le k$ and $0 \le r < m$. Then we have
\begin{align*} 
f^k_{n,A}(\vv) &= \frac1n (A\vv + k\vc) = \frac1n (A\vv + sm\vc) + \frac 1n (r\vc) 
= \frac mn (\frac 1m A\vv + \frac 1m (sm)\vc) + \frac rn \vc \\
&= \frac mn (\frac 1m A\vv + s\vc) + \frac rn\vc 
= f^s_{\frac mn,\frac 1m A}(\vv) + \frac rn\vc.
\end{align*}
Thus when $k_1,\dots,k_{n/m}$ each have the same remainder, i.e. $k_i = (i-1)m+r$, then the set $\{f^{k_1}_{n,A},\dots,f^{k_{n/m}}_{n,A}\}$ differs from the linear $n/m$-valued map $f_{\frac mn,\frac 1m A}$ by the constant $\frac rn \vc$, and thus they are homotopic, proving the second statement.
\end{proof}

Letting $m=n$ in Theorem \ref{linearpartition} gives the following result, which was proved directly by Brown \& Lin in \cite{bl10}:
\begin{cor}
With notation as in Theorem \ref{linearpartition}, the linear map $f_{n,A}:T^q \to D_n(T^q)$ splits if and only if $n$ divides every $l_j$, and in this case $f_{n,A}$ splits into $n$ maps, each homotopic to $f_{1,\frac 1n A}$. 
\end{cor}

For the special case of circle maps (where $q=1$), Theorem \ref{linearpartition} gives:
\begin{cor}
Let $f:S^1 \to D_n(S^1)$ be a circle map of degree $d$. Then $f$ has a nontrivial partition if and only if $n$ and $d$ have a common factor $m>1$. In this case, $f$ partitions into $m$ maps, each an $n/m$-valued map of degree $d/m$.
\end{cor}

No specific example of $n$-valued maps on tori has been presented in the literature other than linear maps. The following example is a torus map which has linear partitions, but is not itself linear. The example is specific, but not very remarkable- it is easy to construct similar examples.

\begin{exl}
Let:
\[ A = \begin{bmatrix} 1 & 0 \\ 3 & 4 \end{bmatrix}, \quad B = \begin{bmatrix} -1 & 0 \\ 1 & 4 \end{bmatrix}, \]
and let $g(t)$ be the linear 2-valued map $g(t) = f_{2,A}(t)$ and $h(t)$ be the following translation of a linear 2-valued map:
\[ h(t) = f_{2,B}(t) + \begin{bmatrix}1/4 \\ 0 \end{bmatrix} \]
where the entries of $h(t)$ are read mod 1.

Then $g$ and $h$ are each 2-valued maps $g,h:T^2 \to D_2(T^2)$, and we claim that $f = \{g,h\}$ is a 4-valued map $f:T^2 \to D_4(T^2)$.

We must show that $f(x)$ is always a set of 4 elements. Clearly $f$ lifts to $\bar f = ( \bar g_1, \bar g_2, \bar h_1,\bar h_2)$, where
\begin{align*} 
\bar g_1(t) &= f_{2,A}^{1}(t), \quad &\bar g_2(t) = f_{2,A}^{2}(t), \\ 
\bar h_1(t) &= f_{2,B}^{1}(t) + \begin{bmatrix}1/4 \\ 0 \end{bmatrix}, \quad &\bar h_2(t) = f_{2,B}^{2}(t) +\begin{bmatrix}1/4 \\ 0 \end{bmatrix}
\end{align*}
and so it suffices to show that $p^4(\bar f(t))$ is always a set of 4 distinct elements, that is, that the 4 elements: $\bar g_1(t),\bar g_2(t),\bar h_1(t),\bar h_2(t)$ never differ by integer vectors.

We already know that this is true for $\bar g_1$ and $\bar g_2$, and for $\bar h_1$ and $\bar h_2$. We must check the others. For $t = (t_1,t_2)$, we have:
\begin{align*}
\bar g_1(t) - \bar h_1(t) &= \frac12 \left(At + \begin{bmatrix}1\\1\end{bmatrix}\right) - \left(\frac12 \left(Bt +  \begin{bmatrix}1\\1\end{bmatrix}\right) + \begin{bmatrix}1/4 \\ 0 \end{bmatrix}\right) \\
&= \frac12(A-B)t - \begin{bmatrix}1/4 \\ 0\end{bmatrix} = \begin{bmatrix} 1 & 0 \\ 1 & 0\end{bmatrix} t - \begin{bmatrix}1/4 \\ 0\end{bmatrix} = \begin{bmatrix} t_1 - 1/4 \\ t_1 \end{bmatrix},
\end{align*}
and this vector cannot be integer valued. Similar computations show
\begin{align*}
\bar g_1(t) - \bar h_2(t) &= \begin{bmatrix} 1 & 0 \\ 1 & 0\end{bmatrix} t - \begin{bmatrix}-1/4 \\ -1/2\end{bmatrix} = \begin{bmatrix} t_1 + 1/4 \\ t_1+1/2 \end{bmatrix} \\
\bar g_2(t) - \bar h_1(t) &=  \begin{bmatrix} 1 & 0 \\ 1 & 0\end{bmatrix} t - \begin{bmatrix}1/4 \\ 1/2\end{bmatrix} = \begin{bmatrix} t_1 - 1/4 \\ t_1-1/2 \end{bmatrix} \\
\bar g_2(t) - \bar h_2(t) &=  \begin{bmatrix} 1 & 0 \\ 1 & 0\end{bmatrix} t - \begin{bmatrix}1/4 \\ 0 \end{bmatrix} = \begin{bmatrix} t_1 - 1/4 \\ t_1\end{bmatrix} 
\end{align*}
and none of these vectors can be integer valued.

Thus $f = \{g,h\}$ is a 4-valued map. In \cite{bl10} it is shown that $N(f_{n,A}) = n |\det(I-\frac1n A)|$, where $I$ is the identity matrix. Then by Corollary \ref{Npartition} we have $N(f) = N(g) + N(h)$, and since $h$ is homotopic to the linear map $f_{2,B}$, we have:
\[ N(f) = N(f_{2,A}) + N(f_{2,B}) = 2|\det(I-\frac12 A)| + 2|\det(I-\frac12 B)| = 1 + 3 = 4. \]
\end{exl}

In the example above, $f$ is nonsplit and nonlinear, but it does partition into submaps which are homotopic to linear maps. It is not clear if all $n$-valued torus maps can be partitioned in this way.

\begin{question}\label{toriquest}
Given any $f:T^q \to D_n(T^q)$, does $f$ have a partition into submaps, each of which is homotopic to a linear map?
\end{question}

\section{Partitions and the Jiang property}\label{jiangsection}
The Jiang subgroup for $n$-valued maps was defined in \cite{bdds20} as follows.
Let $f:X \to D_n(X)$ be an $n$-valued map with a chosen reference lift $\bar f:\tilde X \to F_n(\tilde X,\pi)$.
A homotopy $h \colon X \times I \to D_n(X)$ is a 
\emph{cyclic 
homotopy of $f$} if $h(x,0) = h(x,1) = f(x)$ for 
all $x \in X$.
A cyclic homotopy of $f$ will lift to 
a homotopy  $\bar h \colon \tilde X \times I
\to F_n(\tilde X, \pi)$ with $\bar h(\tilde x,0) =
\bar f(\tilde x)$  and $\bar h(\tilde x,1) = 
(\alpha; \eta)\bar f(\tilde x)$
for some $(\alpha; \eta) \in \pi_1(X)^n \rtimes 
\Sigma_n$.
The \emph{Jiang subgroup 
for $n$-valued maps} $J_n(\bar f) \subseteq \pi_1(X)^n \rtimes 
\Sigma_n$ is the set of all elements
$(\alpha; \eta) \in \pi_1(X)^n \rtimes \Sigma_n$
obtained in this way from cyclic homotopies.

Recall from the work of \cite{bdds20} that $f$ determines a homomorphism $\psi_f:\pi_1(X) \to \pi_1(X)^n \rtimes \Sigma_n$.

The following clarifies the work in \cite{bdds20}, showing that the traditional Jiang-results concerning the equality of the fixed point index are satisfied when the map is irreducible.

\begin{thm}\label{jiangirred}
Let $\psi_f(\pi_1(X)) \subset J_n(\bar f)$, and let $f$ be irreducible. Then all fixed point classes of $f$ have the same index.
\end{thm}
\begin{proof}
Proposition 6.7 of \cite{bdds20} says that if $\psi_f(\pi_1(X)) \subset J_n(\bar f)$ and if there is some $\gamma$ with $\sigma_\gamma(j)=i$, then $p\Fix(\alpha \bar f_i)$ and $p\Fix(\beta \bar f_j)$ have the same index for any $\alpha, \beta \in \pi_1(X)$. Since $f$ is irreducible, such a $\gamma$ does exist for any $i$ and $j$, and thus all fixed point classes have the same index.
\end{proof}

We immediately obtain:
\begin{cor}
Let $\psi_f(\pi_1(X)) \subset J_n(\bar f)$, and let $f = \{ f_1,\dots, f_k\}$ be a partition into irreducibles. Then for each $i$, all fixed point classes of $f_i$ have the same index.
\end{cor} 

\section{Crabb's theory of finite-valued maps}\label{crabbsection}
Parallel to the work of Schirmer \& Brown on $n$-valued Nielsen theory, M.\ C.\ Crabb has developed a Nielsen theory in the context of multivalued maps $f:X\multimap X$ which satisfy $\#f(x)\le n$ for each $x$. The theory of $n$-valued maps is a special case of Crabb's theory, and in many cases Crabb's point of view provides opportunities for simpler computation. Crabb's survey paper \cite{crab18} details the connections between the two theories, which we will briefly review. 

In the most generality, Crabb discusses local maps on ENRs, but for our purposes we will focus on global maps on compact polyhedra. Let $X$ be a compact polyhedron, let $q: \hat X \to X$ be a finite covering map, and $F:\hat X \to X$ be any map. When $\hat X$ is an $n$-fold cover, the map $F:\hat X \to X$ resembles a map $X\to X$, but with the points of the domain each being replicated $n$-times according to the cover $q$. Specifically, $f(x) = F(q^{-1}(x))$ defines a multivalued map $f:X\multimap X$ with $\#f(x)\le n$ for each $x$. Crabb writes the pair $(F,q)$ as a fraction $F/q$, which he calls an ``$n$-valued map.'' To avoid confusion in terminology, we will call $F/q$ a \emph{finite-valued map with cardinality at most $n$}. Typically we will simply call it a \emph{finite-valued map}, as the cardinality bound will usually be clear (and will usually be $n$).

Any $n$-valued map $f:X\to D_n(X)$  can naturally be expressed as a finite-valued map, as follows. The following appears in more generality as Proposition 5.3 of \cite{crab18}.
\begin{thm}
Let $f:X\to D_n(X)$ be an $n$-valued map. Let $\hat X = \Gamma_f$ be the graph of $f$, and let $q:\hat X \to X$ be given by $q(x,y) = x$ and $F:\hat X \to X$ be given by $F(x,y)=y$. Then $F/q$ is a finite-valued map with cardinality at most $n$, and $f = F \circ q^{-1}$ for all $x$. 
\end{thm}
\begin{proof}
It is clear that $F/q$ is finite-valued, we need only show that $f = F \circ q^{-1}$. Take some $x\in X$, and we will show that $f(x)$ and $F(q^{-1}(x))$ are equal as sets. 

First take $y\in f(x)$, which means that $(x,y)\in \Gamma_f$, which implies $(x,y) \in q^{-1}(x)$. Then applying $F$ gives $y \in F(q^{-1}(x))$.

Conversely, take $y\in F(q^{-1}(x))$, so there is some pair $z\in q^{-1}(x) \subset \Gamma_f$ with $F(z)=y$. Since $z \in q^{-1}(x)$, the first coordinate of $z$ must be $x$. Since $F(z)=y$, the second coordinate of $z$ must be $y$. Thus $z=(x,y)$, and so $(x,y)\in \Gamma_f$ and thus $y\in f(x)$.
\end{proof}

When $f:X\to D_n(X)$ and we use the coordinate projections $q:\Gamma_f \to X$ and $F:\Gamma_f \to X$ as above, the finite-valued map $F/q$ will be called \emph{the finite-valued map associated to $f$}.  

The fixed point set of a finite-valued map $F/q$ is defined as a coincidence set:
\[ \Fix(F/q) = \Coin(F,q) = \{ \hat x \in \hat X \mid F(\hat x) = q(\hat x) \}. \]
Note that, when $F/q$ is the finite-valued map associated to an $n$-valued map $f$, we have $\Fix(F/q) \subset \hat X$ while $\Fix(f) \subset X$. In this case we will have $q\Fix(F/q) = \Fix(f)$. Generally for a finite-valued map $F/q$, the sets $\Fix(F/q)$ and $q\Fix(F/q)$ may have a different number of elements.

Crabb defines a fixed point index, Lefschetz number, and Nielsen number for finite-valued maps, and proves that his definitions match the corresponding definitions by Schirmer and Brown when $F/q$ corresponds to an $n$-valued map. 

For any finite-valued map $F/q$, we may consider the set of components of $\hat X$, which we write $\hat X_1,\dots, \hat X_k$. Then $F:\hat X \to X$ naturally decomposes into $k$ component maps $F_i:\hat X_i \to X$. 

From Theorem \ref{graphpartition} we immediately obtain:
\begin{thm}
Let $f:X\to D_n(X)$ be an $n$-valued map, and let $F/q$ be the associated finite-valued map. Then there is an irreducible partition $f=\{f_1,\dots, f_k\}$ if and only if $\hat X$ has $k$ components, and in this case we have $f_i = F_i \circ q^{-1}$ for each $i$. 
\end{thm}

Thus if an $n$-valued map is presented as a finite-valued map $F/q$ (in particular if we are given the structure of the covering $q:\hat X \to X$), then determining the irreducible partition of $f$ becomes obvious: it is simply given by the component maps $F_i/q$. Crabb discusses this decomposition of $F/q$ into components $F_i/q$, and obtains additivity formulas analogous to \eqref{Nparteq}. 

Furthermore, the set of all finite-valued maps on a space $X$ is often easy to classify from this point of view. 
For example if $X$ is a torus, then any connected finite cover of $X$ is also a torus of the same dimension. Thus if $F/q$ is a finite-valued map on $X$, then the cover $\hat X$ naturally decomposes into components $\hat X_i$, each of which are tori, and so $F/q$ has an irreducible partition into finite-valued torus maps $F_i/q$. Each $F_i:\hat X_i \to X$ is a single-valued self-map of a torus, and so can be linearized into a matrix $A_i$, and thus $N(F_i/q) = |\det(I-A_i)|$, and so by \eqref{Nparteq} we have 
\[
 N(F/q) = \sum_{i} |\det(I-A_i)|. 
\]
Thus we have a solution to the problem of computing $N(f)$ when $f:T^q \to D_n(T^q)$ is an $n$-valued torus map, provided that $f$ is specified as a finite-valued map $F/q$ and we are given enough information to deduce the matrices $A_i$ of the component maps $F_i/q$.

The paragraph above would appear to answer Question \ref{toriquest} in the affirmative, but the situation is more subtle. Crabb's work shows that every finite-valued torus map $F/q$ is homotopic through finite-valued maps to one with linear partitions. Even if $F/q$ corresponds to an $n$-valued map, there is no guarantee that the intermediate maps of the homotopy linearizing the components of $F/q$ will be $n$-valued.

We note that although the Nielsen number defined by Crabb's theory agrees with that defined by Schirmer, the minimal number of fixed points may not. Specifically, if $f:X\to D_n(X)$ is an $n$-valued map, let $\MF(f)$ be the minimal number of fixed points of any $n$-valued map homotopic (by $n$-valued homotopy) to $f$, and if $F/q$ is a finite-valued map, let $\MF(F/q)$ be the minimal number of fixed points of any finite-valued map $G/q$ with $G$ homotopic (as a single-valued map) to $F$.

If $F/q$ is the finite-valued map associated to $f$, then both $\MF(f)$ and $\MF(F/q)$ will be bounded below by the Nielsen number $N(f)=N(F/q)$. Any $n$-valued homotopy of $f$ corresponds naturally to a finite-valued homotopy of $F/q$, but the converse may not be true. Thus we will have $N(f) \le \MF(F/q) \le \MF(f)$. 
We suspect that the second inequality can be strict in some cases, but we do not have an example.

\begin{question}
Let $f:X\to D_n(X)$ be an $n$-valued map, and let $F/q$ be the corresponding finite-valued map. Is $\MF(F/q) = \MF(f)$?
\end{question}

The equality in question is related to the Wecken problem for $n$-valued maps: for which spaces and selfmaps $f:X \to D_n(X)$ will we have $N(f)=\MF(f)$? This can be a difficult question even for simple spaces, see \cite{bces19} which proves the Wecken property for any $n$-valued map of the sphere $S^2$. Whenever $N(f)=\MF(f)$, we will automatically have $\MF(F/q) = \MF(f)$.

\end{document}